\documentclass[a4paper,english,fontsize=11pt,parskip=half,abstracton]{scrartcl}
\usepackage{babel}
\usepackage[utf8]{inputenc}
\usepackage[T1]{fontenc}
\usepackage[a4paper,left=21mm,right=21mm,top=30mm,bottom=30mm]{geometry}
\usepackage{amsmath}
\usepackage{amsthm}
\usepackage{amssymb}
\usepackage{enumerate}
\usepackage{aliascnt}
\usepackage[bookmarks=false,
            pdftitle={On the blockwise modular isomorphism problem},
            pdfauthor={Gabriel Navarro, Benjamin Sambale},
            pdfkeywords={modular isomorphism problem, blocks, defect groups},
            pdfstartview={FitH}]{hyperref}

\newtheorem{Theorem}{Theorem}[section]
\newaliascnt{Lemma}{Theorem}
\newtheorem{Lemma}[Lemma]{Lemma}
\aliascntresetthe{Lemma}
\newaliascnt{Proposition}{Theorem}
\newtheorem{Proposition}[Proposition]{Proposition}
\aliascntresetthe{Proposition}
\newaliascnt{Corollary}{Theorem}
\newtheorem{Corollary}[Corollary]{Corollary}
\aliascntresetthe{Corollary}
\newaliascnt{Question}{Theorem}
\newtheorem{Question}[Question]{Question}
\aliascntresetthe{Question}

\numberwithin{equation}{section}

\allowdisplaybreaks[1]

\renewcommand{\phi}{\varphi}
\newcommand{\C}{\operatorname{C}}

\newcommand{\Z}{\operatorname{Z}}
\newcommand{\K}{\operatorname{K}}
\newcommand{\ZZ}{\mathbb{Z}}
\newcommand{\CC}{\mathbb{C}}

\newcommand{\FF}{\mathbb{F}}
\newcommand{\cohom}{\operatorname{H}}
\newcommand{\Aut}{\operatorname{Aut}}

\newcommand{\Out}{\operatorname{Out}}
\newcommand{\pcore}{\mathrm{O}}
\newcommand{\GL}{\operatorname{GL}}
\newcommand{\SL}{\operatorname{SL}}

\newcommand{\End}{\operatorname{End}}

\newcommand{\Irr}{\operatorname{Irr}}
\newcommand{\IBr}{\operatorname{IBr}}

\newcommand{\F}{\mathrm{F}}

\newcommand{\Syl}{\operatorname{Syl}}

\newcommand{\Ker}{\operatorname{Ker}}
\newcommand{\lcm}{\operatorname{lcm}}

\mathchardef\ordinarycolon\mathcode`\:  
 \mathcode`\:=\string"8000
 \begingroup \catcode`\:=\active
   \gdef:{\mathrel{\mathop\ordinarycolon}}
 \endgroup

\title{On the blockwise modular\\ isomorphism problem}
\author{Gabriel Navarro\footnote{Department of Mathematics, Universitat de Val\`encia, 46100 Burjassot,
        Spain, \href{mailto:gabriel@uv.es}{gabriel@uv.es}} \ and Benjamin Sambale\footnote{Fachbereich Mathematik, TU Kaiserslautern, 67653 Kaiserslautern, Germany, \href{mailto:sambale@mathematik.uni-kl.de}{sambale@mathematik.uni-kl.de}}}
\date{\today}

\begin{document}
\frenchspacing
\maketitle
\begin{abstract}\noindent
As a generalization of the modular isomorphism problem we study the behavior of defect groups under Morita equivalence of blocks of finite groups over algebraically closed fields of positive characteristic. We prove that the Morita equivalence class of a block $B$ of defect at most $3$ determines the defect groups of $B$ up to isomorphism. In characteristic $0$ we prove similar results for metacyclic defect groups and $2$-blocks of defect $4$.
In the second part of the paper we investigate the situation for $p$-solvable groups $G$. Among other results we show that the group algebra of $G$ itself determines if $G$ has abelian Sylow $p$-subgroups.
\end{abstract}

\textbf{Keywords:} modular isomorphism problem, Morita equivalence, blocks, defect groups\\
\textbf{AMS classification:} 20C05, 20C20

\section{Introduction}
The longstanding \emph{modular isomorphism problem} asks if finite $p$-groups $G$ and $H$ are isomorphic whenever their group algebras $\FF_pG$ and $\FF_pH$ over the field of $p$ elements are isomorphic. In the study of the
global/local conjectures in representation theory we do not often encounter
isomorphism of algebras, but a much weaker situation: \emph{Morita equivalences} of categories of blocks. And yet we do not know the answer to the following fundamental question:

\begin{Question}\label{Q1}
Let $B$ be a block of a finite group $G$ with respect to an algebraically closed field $F$ of characteristic $p>0$. Does the Morita equivalence class of $B$ determine the isomorphism type of a defect group $D$ of $B$?
\end{Question}

If $G$ and $H$ are $p$-groups such that $\FF_pG\cong\FF_pH$, then 
\[B_0(FG)=FG\cong F\otimes_{\FF_p}\FF_pG\cong F\otimes_{\FF_p}\FF_pH\cong FH=B_0(FH)\]
where $B_0$ denotes the principal block. Since isomorphic blocks are certainly Morita equivalent, \autoref{Q1} is actually an extension of the modular isomorphism problem.

This paper started as a systematic search for a counterexample to \autoref{Q1} which, somewhat surprisingly, does not appear to have been done before. Instead of a counterexample, we have been able to answer \autoref{Q1} positively
for blocks of small defect, and in some other situations. This has required to combine a good deal of previous 
theoretical results by many authors together with some new ideas, and heavy use of computers.
Of course, we are bound by the modular isomorphism problem which up to date, has only been answered for $p$-groups of order at most $p^5$ ($2^9$ and $3^6$ for $p=2$ and $p=3$ respectively, see \cite[Introduction]{EickKonovalov}). 
It does not seem unreasonable to think that there are counterexamples to both the modular isomorphism
problem and \autoref{Q1}, but due to the fact that group algebras are exponentially bigger than groups, perhaps these are not within the reach of computers yet. 

In the situation of \autoref{Q1} the following invariants are known to be determined by the Morita equivalence class of $B$:
\begin{itemize}
\item The \emph{Cartan matrix} $C$ of $B$ up to the order of the simple modules. In particular, the number $l(B)$ of irreducible Brauer characters of $B$ is determined.
\item The \emph{order} $|D|$ of $D$ is the largest elementary divisor of $C$. In particular, the \emph{defect} of $B$ is determined.
\item The \emph{exponent} $\exp(D)$ of $D$ is determined via Külshammer ideals (see \cite[(78)]{KProgress}).
\item The \emph{rank} $r(D)$ of $D$ is determined via the complexity of the indecomposable modules (see \cite[Corollary~4]{AlpEv81} or Bessenrodt~\cite[Proposition~2.1]{Bessenrodtnewinv}).
\item It is determined if $D$ is dihedral, semidihedral or quaternion via the representation type (here $p=2$).
\item If $D$ is known to be abelian (or even hamiltonian), then the isomorphism type of $D$ is determined (see Bessenrodt~\cite[Theorems~2.1 and 7.4]{Bessenrodtabel}).
\item The isomorphism type of the center $\Z(B)$ of $B$. In particular, the number $k(B)$ of ordinary irreducible characters in $B$ is determined via $k(B)=\dim_F\Z(B)$. 
\end{itemize}

If $B$ is nilpotent, then by Puig's theorem, $B$ is Morita equivalent to $FD$ (see \cite[Theorem~1.30]{habil}). 
Since $FD$ is a basic algebra, $B$ even determines the isomorphism type of $FD$. So we are down to the modular isomorphism problem.

Dade has constructed non-isomorphic finite groups $G$ and $H$ such that $FG\cong FH$ (see \cite[Theorem~14.2.2]{PassmanASGR}). Both groups have the form $D\rtimes K$ for some $p'$-group $K$ (they differ only by the action of $K$ on $D$). This shows that the \emph{fusion system} of a block is in general not determined by the Morita equivalence class.

If we work over a complete discrete valuation ring $\mathcal{O}$ with $\mathcal{O}/J(\mathcal{O})\cong F$ instead, then also the \emph{decomposition matrix} $Q$ of $B$ is given by the Morita equivalence class. The \emph{heights} of the ordinary irreducible characters can be extracted from $Q$ by Brauer's theory of contributions (see \cite[Proposition~1.36]{habil}). Brauer's (still open) height zero conjecture would imply that $D$ is abelian if and only if all characters have height $0$. By work of Kessar-Malle, $D$ is non-abelian whenever there are characters of positive height (see \cite[Theorem~7.14]{habil}).
Over the valuation ring, Puig~\cite[Theorem~8.2]{PuigLocalBuch} showed that $B$ is nilpotent if and only if $B$ and $\mathcal{O}D$ are Morita equivalent. 
If so, then the isomorphism type of $D$ is uniquely determined by Roggenkamp-Scott~\cite{RoggenkampScott}. A refined version of the block modular isomorphism problem over valuation rings has been introduced by Scott~\cite{Scott}.

The paper is organized as follows: In the next section we prove some general results and give an affirmative answer to \autoref{Q1} for blocks of defect at most $3$ and $2$-blocks of defect $4$. Moreover, we provide a solution for metacyclic defect groups. In Section~3 we restrict ourselves to $p$-solvable groups $G$. 
Here we show first that principal blocks uniquely determine a group algebra up to isomorphism. Next we prove in \autoref{abeldef} that the group algebra $FG$ determines if $G$ has abelian Sylow $p$-subgroups. This appears to be an open problem for arbitrary finite groups $G$.  
As applications we solve \autoref{Q1} for $2$-blocks of defect $5$ and $3$-blocks of defect $4$ (of $p$-solvable groups). 

\section{Blocks of defect 3 and metacyclic defect groups}

Our notation is fairly standard. The Jacobson radical of a module $M$ (or a ring) is denoted by $J(M)$ and the Loewy length is $LL(M)$. The symbols $C_n^m$, $D_{2n}$, $Q_{2^n}$, $SD_{2^n}$, $S_n$, $A_n$ and $p^{1+2}_+$ represent the abelian group of type $(n,\ldots,n)$ ($m$ times), the dihedral group of order $2n$, the quaternion group of order $2^n$, the semidihedral group of order $2^n$, the symmetric group of degree $n$, the alternating group of degree $n$, and the extraspecial group of order $p^3$ and exponent $p$.

In the following we are mostly interested in the situation over $F$, but work over $\mathcal{O}$ from time to time to make things easier.
Our first result generalizes the main result of \cite{Baginski}.

\begin{Proposition}\label{coexp1}
Let $B$ be a $p$-block of $FG$ with defect group $D$ such that $|D|/\exp(D)\le p$. Then the Morita equivalence class of $B$ determines $D$ up to isomorphism.
\end{Proposition}
\begin{proof}
We may assume that $\exp(D)=p^{d-1}$ where $d$ is the defect of $B$. Suppose first that $p=2$. If $B$ has wild representation type, then $D\cong C_{2^{d-1}}\times C_2$ with $d\ge 3$ or 
\[D\cong\langle x,y\mid x^{2^{d-1}}=y^2=1,\,yxy^{-1}=x^{1+2^{d-2}}\rangle\] 
with $d\ge 4$. In both cases $B$ is nilpotent (see \cite[Theorem~8.1]{habil}) and the defect groups can be distinguished, since one is abelian and the other is not. Now we assume that $B$ has tame representation type, i.\,e. $D$ is dihedral (including the Klein four-group), semidihedral or quaternion. Since quaternion groups have rank $1$, it suffices to consider dihedral groups and semidihedral groups (here $d\ge 4$). 
If $l(B)=1$, then $B$ is nilpotent and $D$ is determined by \cite[Theorem~1]{Baginski} for instance.
By comparing $k(B)$, we may assume that $l(B)=2$. By Erdmann~\cite{Erdmann}, it is known that in the dihedral case the hearts of the projective indecomposable modules are always uniserial or decomposable. In the semidihedral case the contrary happens.

Now let $p>2$. We need to distinguish the defect groups $C_{p^{d-1}}\times C_p$ and $C_{p^{d-1}}\rtimes C_p$ for $d\ge 3$. Suppose first that $D\cong C_{p^{d-1}}\rtimes C_p$. Then by a theorem of Watanabe, $l(B)$ divides $p-1$ and 
\[k(B)=\frac{p^{d-1}+p^{d-2}-p^{d-3}-p}{l(B)}+l(B)p\]
(see \cite[Theorems~1.33 and 8.13]{habil}). 
Let $A$ be a block with defect group $\langle x\rangle\times\langle y\rangle\cong C_{p^{d-1}}\times C_p$ and $l(A)=l(B)$. We may assume that the inertial quotient $E$ of $A$ stabilizes $\langle x\rangle$ and $\langle y\rangle$. In order to compute $k(A)$ we count subsections for $A$ (see \cite[p. 11]{habil} for a definition). Let $d_x\mid p-1$ be the order of the image of $E\to\Aut(\langle x\rangle)$. Similarly, we define $d_y$. Note that $\lcm(d_x,d_y)\mid |E|\mid d_xd_y$. There are $(p^{d-1}-1)/d_x$ non-trivial subsections of the form $(x^i,a_{x^i})$ up to conjugation. By another result of Watanabe, we have $l(a_{x^i})=l(a_x)$ for all $i\ne 0$ (see \cite[Theorem~1.39]{habil}). The block $a_x$ dominates a unique block $\overline{a_x}$ of $\C_G(x)/\langle x\rangle$ with cyclic defect group $D/\langle x\rangle\cong\langle y\rangle$ and inertial index $|E|/d_x$. Consequently, $l(a_{x^i})=l(a_x)=|E|/d_x$ by Dade's theory of blocks with cyclic defect groups (see \cite[Theorem~8.6]{habil}). 
Similarly, there are $(p-1)/d_y$ non-trivial subsections of the form $(y^i,a_y)$ with $l(a_y)=|E|/d_y$. Finally, there are $(p^{d-1}-1)(p-1)/|E|$ non-trivial subsections $(z,a_z)$ with $z\in \langle x,y\rangle\setminus(\langle x\rangle\cup\langle y\rangle)$. Here, $a_z$ has inertial index $1$ for all $z$ and it follows that $l(a_z)=1$. Now Brauer's formula implies
\begin{align*}
k(A)=l(A)+\frac{(p^{d-1}-1)|E|}{d_x^2}+\frac{(p-1)|E|}{d_y^2}+\frac{(p^{d-1}-1)(p-1)}{|E|}\ge l(B)+\frac{p^d-1}{p-1}
\end{align*}
(see \cite[Theorem~1.35]{habil}).
Note that $k(A)-k(B)$ is a concave function in $l(A)$ which assumes its minimum on $l(A)=1$ or $l(A)=p-1$. 
An easy computation yields $k(A)>k(B)$. Hence, $A$ and $B$ are not Morita equivalent.
\end{proof}

The following theorem extends \cite[Proposition~2.2]{Bessenrodtnewinv}.

\begin{Theorem}\label{defect3}
Let $B$ be a block of $FG$ with defect at most $3$. Then the Morita equivalence class of $B$ determines the defect group of $B$ up to isomorphism.
\end{Theorem}
\begin{proof}
By \autoref{coexp1}, we may assume that $B$ has a defect group $D$ of order $p^3$ and exponent $p$. If $p=2$, then $D$ is elementary abelian. If $p>2$, then there are two possible defect groups which differ by their rank. Hence, in any case $D$ is determined up to isomorphism.
\end{proof}

In order to deal with the $2$-blocks of defect $4$, we need a lemma about perfect isometries.

\begin{Lemma}\label{perfect}
Let $G=S_4\times C_2$ and $H=A_4\rtimes C_4$ where $C_4$ acts as a transposition on $A_4$. Then $\mathcal{O}G$ and $\mathcal{O}H$ are not perfectly isometric.
\end{Lemma}
\begin{proof}
Recall from \cite[Section 4.D]{BroueEqui} that a perfect isometry is a bijection $I:\Irr(G)\to\Irr(H)$ with signs $\epsilon:\Irr(G)\to\{\pm1\}$ such that the map
\[\mu:G\times H\to\mathcal{O},\qquad (g,h)\to\sum_{\chi\in\Irr(G)}{\epsilon(\chi)\chi(g)I(\chi)(h)}\]
satisfies the following properties:
\begin{enumerate}[(a)]
\item if exactly one of $g$ and $h$ is $p$-regular, then $\mu(g,h)=0$;
\item $\mu(g,h)/\lvert\C_G(g)\rvert,\mu(g,h)/\lvert\C_H(h)\rvert\in\mathcal{O}$ for all $g\in G$, $h\in H$.
\end{enumerate}
Let $\Z(G)=\langle z\rangle$ and $H=A_4\rtimes\langle h\rangle$. We need the following columns of the character tables of $G$ and $H$ (here $i=\sqrt{-1}$):
\[
\begin{array}{c|ccccc}
G&1&(12)&(1234)&(123)&(12)z\\\hline
\lambda_1&1&1&1&1&1\\
\lambda_2&1&1&1&1&-1\\
\lambda_3&1&1&-1&1&-1\\
\lambda_4&1&1&-1&1&1\\
\chi_1&2&.&.&-1&.\\
\chi_2&2&.&.&-1&.\\
\psi_1&3&1&-1&.&1\\
\psi_2&3&1&-1&.&-1\\
\psi_3&3&-1&1&.&-1\\
\psi_4&3&-1&1&.&1
\end{array}
\qquad
\begin{array}{c|ccccc}
H&1&(12)(34)&h&h^2&(123)h^2\\\hline
\widetilde{\lambda}_1&1&1&1&1&1\\
\widetilde{\lambda}_2&1&1&-1&1&1\\
\widetilde{\lambda}_3&1&1&-i&-1&-1\\
\widetilde{\lambda}_4&1&1&i&-1&-1\\
\widetilde{\chi}_1&2&2&.&-2&1\\
\widetilde{\chi}_2&2&2&.&2&-1\\
\widetilde{\psi}_1&3&-1&-1&3&.\\
\widetilde{\psi}_2&3&-1&1&3&.\\
\widetilde{\psi}_3&3&-1&-i&-3&.\\
\widetilde{\psi}_4&3&-1&i&-3&.
\end{array}
\]
By \cite[Theorem~4.11]{BroueEqui}, $I$ preserves the heights of the characters. In particular, $\{\chi_1,\chi_2\}$ maps to $\{\widetilde{\chi}_1,\widetilde{\chi}_2\}$.
Now $\mu((123),(12)(34))=0=\mu((123),h^2)$ shows that $\{\lambda_1,\ldots,\lambda_4\}$ maps to $\{\widetilde{\lambda}_1,\ldots,\widetilde{\lambda}_4\}$. 
Moreover, $\mu((123),h)=0$ implies $\epsilon(I^{-1}(\widetilde{\lambda}_3))=\epsilon(I^{-1}(\widetilde{\lambda}_4))$ and $\epsilon(I^{-1}(\widetilde{\lambda}_1))=\epsilon(I^{-1}(\widetilde{\lambda}_2))$. Since $\mu((123),(123)h^2)=0$, we also have $\epsilon(\lambda_1)=\ldots=\epsilon(\lambda_4)$. 
From 
\[\mathcal{O}\ni\frac{1}{\lvert\C_G((12))\rvert}\bigl(\mu((12),(123)h^2)+\mu((12)z,(123)h^2)\bigr)=\frac{\epsilon(\lambda_1)}{4}\bigl(I(\lambda_1)((123)h^2)-I(\lambda_3)((123)h^2)\bigr)\] 
we obtain that $I(\{\lambda_1,\lambda_3\})\in\{\{\widetilde{\lambda}_1,\widetilde{\lambda}_2\},\{\widetilde{\lambda}_3,\widetilde{\lambda}_4\}\}$. This yields the final contradiction
\[\frac{1}{\lvert\C_H(h)\rvert}\bigl(\mu((12),h)+\mu((1234),h)\bigr)=\epsilon(\lambda_1)\frac{1-i}{2}\notin\mathcal{O}.\qedhere\]
\end{proof}

\begin{Theorem}\label{2def4}
Let $B$ be a block of $\mathcal{O}G$ with defect group $D$ of order $16$. Then the Morita equivalence class of $B$ determines $D$ up to isomorphism.
\end{Theorem}
\begin{proof}
For most of the proof we argue over $F$. 
By \autoref{coexp1}, we may assume that $D$ has exponent $4$. 
Suppose in addition that $r(D)=2$. Then $D\in\{C_4^2,\,C_4\rtimes C_4,\,Q_8\times C_2,\,D_8*C_4\}$ where $D_8*C_4$ denotes a central product. If $l(B)=1$, then $B$ is always nilpotent by \cite[Theorems~8.1, 9.28 and 9.18]{habil}. In this case $D$ is determined by \cite[Lemma~14.2.7]{PassmanASGR}. Next let $l(B)>1$. Then $l(B)=3$ and $D\not\cong C_4\rtimes C_4$. If $k(B)=8$, then $D\cong C_4^2$. It remains to deal with the last two groups. In \cite[proof of Proposition~13]{KS2} we have computed the Loewy length $LL(\Z(B))$ of $\Z(B)$. It turns out that $LL(\Z(B))=4$ if and only if $D\cong Q_8\times C_2$.
  
Now let $r(D)=3$. Then $D\in\{C_4\times C_2^2,\,D_8\times C_2,\,M\}$ where 
\[M:=\langle x,y\mid x^4=y^2=[x,y]^2=[x,x,y]=[y,x,y]=1\rangle\cong C_2^2\rtimes C_4\] 
is minimal non-abelian. Again $B$ is nilpotent if and only if $l(B)=1$ by \cite[Theorems~9.7 and 12.4]{habil}. The defect group is then determined via \cite[Lemma~14.2.7]{PassmanASGR}. Let $l(B)>1$. We have $k(B)=16$ if and only if $D\cong C_4\times C_2^2$. We are left with the last two cases where $k(B)=10$. Here $l(B)=3$ implies $D\cong D_8\times C_2$. Thus, let $l(B)=2$. 
If $D\cong M$, then $B$ is perfectly isometric to $\mathcal{O}[A_4\rtimes C_4]$ as shown in \cite[Theorem~9]{Sambalemna3}. For the group $D\cong D_8\times C_2$ one can construct a perfect isometry between $B$ and $\mathcal{O}[S_4\times C_2]$ (see \cite[proof of Proposition~13]{KS2}, this relies on the computation of generalized decomposition numbers up to basic sets). By \autoref{perfect}, these groups are not perfectly isometric. Since Morita equivalence over $\mathcal{O}$ implies derived equivalence and derived equivalence implies the existence of a perfect isometry, the two defect groups can be distinguished.
\end{proof}

For the last argument in the proof above we need to work over $\mathcal{O}$, since it can be shown that the centers of $F[S_4\times C_2]$ and $F[A_4\rtimes C_4]$ are isomorphic. In fact there is an isomorphism preserving the Reynolds ideal (an invariant under perfect isometries, see \cite[Proposition~6.8]{BHHKM}). On the other hand, $S_4\times C_2$ and $A_4\rtimes C_4$ have different Külshammer ideals $T_1^\perp$. It has been asked in \cite[Question~6.7]{BHHKM} whether perfect isometries also preserve Külshammer ideals. In general this is not the case as can be seen already from $\mathcal{O}D_8$ and $\mathcal{O}Q_8$. 
We like to mention further that there are also non-solvable groups like $S_5\times C_2$ and $A_5\rtimes C_4$ having blocks with the same properties. These blocks are not Morita equivalent to those of $S_4\times C_2$ and $A_4\rtimes C_4$.

Over $\mathcal{O}$ we are able to prove a blockwise version of \cite{SandlingMeta} which generalizes \autoref{coexp1}.

\begin{Theorem}
Let $B$ be a block of $\mathcal{O}G$ with metacyclic defect group $D$. Then the Morita equivalence class of $B$ determines $D$ up to isomorphism.
\end{Theorem}
\begin{proof}
Since we are working over $\mathcal{O}$, we may assume that $B$ is not nilpotent (see introduction). We may also assume that $D$ is non-abelian, because the height zero conjecture is known to hold for metacyclic defect groups (see~\cite[Corollary~8.11]{habil}). Finally by \autoref{coexp1}, we may assume that $D/\exp(D)>p$. Now it follows from \cite[Theorem~8.1]{habil} that $p>2$. By \cite[Theorem 8.8]{habil}, we have
\[D=\langle x,y\mid x^{p^m}=y^{p^n}=1,\ yxy^{-1}=x^{1+p^l}\rangle\cong C_{p^m}\rtimes C_{p^n}\]
where $0<l<m$ and
\[k(B)=\Bigl(\frac{p^l+p^{l-1}-p^{2l-m-1}-1}{l(B)}+l(B)\Bigr)p^n.\]
Moreover, \cite[Theorem~1.33]{habil} implies that the elementary divisors of the Cartan matrix of $B$ are $\lvert\C_D(E)\rvert=p^n$ and $|D|=p^{n+m}$ where $E$ denotes the inertial quotient of $B$ (see also \cite[proof of Theorem~8.8]{habil}). Hence, the Morita equivalence class of $B$ determines $m$ and $n$. It remains to determine $l$. Since $B$ determines $k(B)$, it also determines $p^l+p^{l-1}-p^{2l-m-1}$. It follows from $l<m$ that $2l-m-1<l-1$ and $p^l<p^l+p^{l-1}-p^{2l-m-1}<p^{l+1}$. In this way we obtain $l$.
\end{proof}

For later use, we collect some invariants of group algebras.

\begin{Proposition}\label{lem}
The isomorphism type of $FG$ determines the following:
\begin{enumerate}[(i)]
\item\label{a} $F[G/\pcore^p(G)]$ and $F[G/G']$ up to isomorphism;
\item\label{b} $\lvert\pcore^q(G)\rvert$ for every prime $q$;
\item\label{b2} if $G$ has a normal Sylow $p$-subgroup;
\item\label{b3} if $G/\pcore_{p'p}(G)$ is abelian;
\item\label{b4} if $G=\pcore_{p'pp'}(G)$;
\item\label{c} The number of conjugacy classes in $\{g^{p^i}:g\in G\}$ for every $i\ge 0$;
\item\label{d} The number of conjugacy classes of maximal elementary abelian $p$-subgroups of $G$ of given rank.
\end{enumerate}
\end{Proposition}
\begin{proof}
Let $\sigma:FG\to FH$ be an isomorphism of $F$-algebras and let $\nu:FH\to F$ be the augmentation map where $H$ is another finite group.
Since every $g\in G$ is a unit in $FG$, we have $\nu(\sigma(g))\ne 0$.
Hence, the $F$-linear map $\widetilde{\sigma}:FG\to FH$ given by $\widetilde{\sigma}(g):=\nu(\sigma(g))^{-1}\sigma(g)$ for $g\in G$ is also an isomorphism. Thus, after replacing $\sigma$ by $\widetilde{\sigma}$ we may assume that $\nu(\sigma(g))=1$ for $g\in G$. 

Let $g\in G$ be a $p$-regular element, and let $\phi:FH\to F[H/\pcore^p(H)]$ be the natural epimorphism. Then $\phi(\sigma(g))-1$ lies in the augmentation ideal of $F[H/\pcore^p(H)]$. Since $H/\pcore^p(H)$ is a $p$-group, $\phi(\sigma(g))-1$ is nilpotent. Thus, there exists $n\in\mathbb{N}$ such that $\phi(\sigma(g))^{p^n}-1=(\phi(\sigma(g))-1)^{p^n}=0$. It follows that $\phi(\sigma(g))$ is a $p$-element, but also a $p'$-element. Consequently, $\phi(\sigma(g))=1$. Since $\pcore^p(G)$ is generated by the $p$-regular elements, we have $x-1\in\Ker(\phi\circ\sigma)$ for every $x\in\pcore^p(G)$. 
On the other hand, the elements $x-1$ generate the kernel of the natural epimorphism $FG\to F[G/\pcore^p(G)]$ as an ideal. This shows that the map $\Phi:F[G/\pcore^p(G)]\to F[H/\pcore^p(H)]$, $g\pcore^p(G)\mapsto \phi(\sigma(g))\pcore^p(H)$ for $g\in G$ is a well-defined epimorphism. In particular, $|G/\pcore^p(G)|\ge|H/\pcore^p(H)|$. By symmetry, we also have $|H/\pcore^p(H)|\ge|G/\pcore^p(G)|$ and $\Phi$ is an isomorphism. Altogether we have shown that $FG$ determines $F[G/\pcore^p(G)]$.

Now let $g=[x,y]\in G$ be a commutator, and let $\phi:FH\to F[H/H']$ be the natural epimorphism. Since $F[H/H']$ is commutative, we have $\phi(\sigma(g))=[\phi(\sigma(x)),\phi(\sigma(y))]=1$. This shows $g-1\in\Ker(\phi\circ\sigma)$ for every $g\in G'$. As above we obtain that $FG$ determines $F[G/G']$. This finishes the proof of \eqref{a}. Part \eqref{b} was done in \cite[Theorem~1]{CosseyHawkes}. 
References for \eqref{b2}--\eqref{b4} can be found in \cite[Proposition~2.1]{HanakiKoshitani}.
The number of conjugacy classes in $\{g^{p^i}:g\in G\}$ coincides with the dimension of $i$-th Külshammer ideal (see \cite[(38)]{KProgress}). This settles \eqref{c}.
Finally, part \eqref{d} follows from work of Quillen~\cite{Quillen}. 
\end{proof}

As mentioned in the introduction, $F[G/\pcore^p(G)]\cong F[H/\pcore^p(H)]$ often implies $G/\pcore^p(G)\cong H/\pcore^p(H)$. It is an open question whether $FG$ also determines the normality of Sylow $q$-subgroups where $q\ne p$ (cf. \cite[Question~12]{Navarro03}).

Our next result extends a known fact for $p$-group algebras to certain block algebras.

\begin{Proposition}\label{normal}
Let $B$ be a block of $FG$ with normal defect group $D$. Then the Morita equivalence class of $B$ determines the dimensions of the simple modules of $B$ up to a common scalar. Moreover, the Morita equivalence class determines the order of the Jennings subgroups $J_i(D)$ where $J_1(D):=D$ and $J_n(D):=[D,J_{n-1}(D)]J_{\lceil\frac{n}{p}\rceil}(D)^p$ for $n>1$. In particular, the minimal number of generators of $D$ is determined and $LL(B)=LL(FD)$.
\end{Proposition}
\begin{proof}
By a result of Külshammer, we may assume that $D$ is a Sylow $p$-subgroup of $G$ (see \cite[Theorem~1.19]{habil}).
By the Schur-Zassenhaus Theorem, we get $G=D\rtimes Q$ where $Q$ is a $p'$-group.
Then the irreducible Brauer characters of $B$ can be identified with irreducible characters of $Q$. In particular, the degree vector $v:=(\phi(1):\phi\in\IBr(B))$ consists of $p'$-numbers. It follows from \cite[Corollary~10.14]{Navarro} that $v$ is an eigenvector of the Cartan matrix $C$ of $B$ (corresponding to the eigenvalue $|D|$). Since $C$ is non-negative and indecomposable, the Perron-Frobenius theory shows that $C$ has only one positive eigenvector up to scalar multiplication (see \cite[Theorem~1.4.4]{Minc}). 
This implies the first claim. The Morita equivalence determines the decomposition of the radical layers $J(B)^i/J(B)^{i+1}$ into simple modules. Hence, by the first part of the proof, we know the dimensions of $J(B)^i/J(B)^{i+1}$ up to a common scalar. Let $e\in\Z(FG)$ be the block idempotent of $B$. Then $e\in\Z(FQ)$. It is well-known that
\begin{equation}\label{rad}
J(B)^i=J(eFGe)^i=eJ(FG)^ie=eFGeJ(FD)^i=eFQeJ(FD)^i.
\end{equation}
If $x_1,\ldots,x_n\in FQ$ is an $F$-basis of $eFQe$, and $y_1,\ldots,y_m\in FD$ is an $F$-basis of $J(FD)^i$, then the elements $x_iy_j$ form a basis of $J(B)^i$. Consequently, $\dim J(B)^i/\dim J(FD)^i$ does not depend on $i$. Hence, we have shown that the Morita equivalence class of $B$ determines the dimensions of the Loewy layers of $FD$. On the other hand, these dimensions determine the orders of the Jennings subgroups by \cite[Theorem~VIII.2.10]{Huppert2}.
The last two claims follow from $J_2(D)=\Phi(D)$ and \eqref{rad}. 
\end{proof}

\section{Blocks of $p$-solvable groups}

In the following we restrict ourselves to $p$-solvable groups $G$. Then the structure of $G$ can be vastly reduced. The following proposition is certainly well-known, but we were unable to find a reference where the condition $\pcore_{p'}(G)\le G'$ is proved. Therefore, we provide a proof for the convenience of the reader.

\begin{Proposition}\label{main}
Every $p$-block $B$ of a $p$-solvable group is Morita equivalent to a faithful block of a $p$-solvable group $G$ such that the following holds:
\begin{enumerate}[(i)]
\item the defect groups of $B$ are isomorphic to the Sylow $p$-subgroups of $G$;
\item $Z:=\pcore_{p'}(G)\le \Z(G)\cap G'$ and $Z$ is cyclic.
\end{enumerate}
\end{Proposition}
\begin{proof}
By Külshammer~\cite{Kpsolv}, $B$ is Morita equivalent to a twisted group algebra $F_\gamma L$ where $L$ is a $p$-solvable group such that
\begin{itemize}
\item the defect groups of $B$ are isomorphic to the Sylow $p$-subgroups of $L$;
\item $\pcore_{p'}(L)=1$.
\end{itemize}
If $\gamma=1$, then the claim follows with $G=L$. Thus, we may assume that $\gamma\ne 1$. It is well-known that $\cohom^2(L,F^\times)=\pcore_{p'}(\cohom^2(L,\mathbb{C}^\times))$ (see \cite[Propsition~2.1.14]{Karpilovsky}). Hence, there exists a Schur cover $\widehat{L}$ of $L$ such that $\cohom^2(L,F^\times)\cong W\le\Z(\widehat{L})\cap \widehat{L}'$ and $\widehat{L}/W\cong L$. Choose preimages $\widehat{x}\in\widehat{L}$ for $x\in L$ such that $\widehat{1}=1$. By Maschke's Theorem, $FW=\bigoplus_{\chi\in\Irr(W)}{Fe_\chi}$ for pairwise orthogonal idempotents $e_\chi$. Since $W\le\Z(\widehat{L})$, we have $e_\chi\in\Z(F\widehat{L})$ and
\[F\widehat{L}=\bigoplus_{x\in L}{\widehat{x}FW}=\bigoplus_{\chi\in\Irr(W)}\bigoplus_{x\in L}F\widehat{x}e_\chi.\] 
Let $\alpha\in\Z^2(L,W)$ be the cocycle defined by $\widehat{x}\widehat{y}=\alpha(x,y)\widehat{xy}$ for $x,y\in L$. Then $\widehat{x}e_\chi\cdot\widehat{y}e_\chi=\widehat{xy}\alpha(x,y)e_\chi$. 
Every element $w\in W$ can be written in the form $w=\sum_{\chi\in\Irr(W)}\chi(w)e_\chi$. In particular, $\alpha(x,y)e_\chi=\chi(\alpha(x,y))e_\chi$.
It follows that $\bigoplus_{x\in L}F\widehat{x}e_\chi$ is isomorphic to the twisted group algebra $F_{\Gamma(\chi)}L$ where $\Gamma:\Irr(W)\to\cohom^2(L,F^\times)$, $\chi\mapsto\chi\circ\alpha$. 

We show that $\Gamma$ is surjective. Since $\lvert\Irr(W)\rvert=|W|=\lvert\cohom^2(L,F^\times)\rvert$, it suffices to show that $\Gamma$ is injective. Let $\chi\in\Ker(\Gamma)$. Then there exists a map $\rho:L\to F^\times$ such that $\chi\circ\alpha=\partial\rho$. We have $\rho(1)=\partial\rho(1,1)=\chi(\alpha(1,1))=\chi(1)=1$. Let $\widehat{\chi}:\widehat{L}\to F^\times$ with $\widehat{\chi}(\widehat{x}w):=\rho(x)\chi(w)$ for $x\in L$ and $w\in W$. Obviously, $\widehat{\chi}$ extends $\chi$. For $x,y\in L$ and $w,z\in W$ we have
\begin{align*}
\widehat{\chi}(\widehat{x}w\cdot\widehat{y}z)&=\widehat{\chi}(\widehat{xy}\alpha(x,y)wz)=\rho(xy)\chi(\alpha(x,y))\chi(w)\chi(z)=\rho(xy)\rho(x)\rho(y)\rho(xy)^{-1}\chi(w)\chi(z)\\
&=\rho(x)\chi(w)\rho(y)\chi(z)=\widehat{\chi}(\widehat{x}w)\widehat{\chi}(\widehat{y}z).
\end{align*}
Hence, $\widehat{\chi}$ is a homomorphism and $\widehat{L}/\Ker(\widehat{\chi})\le F^\times$ is abelian. Consequently, $W\le\widehat{L}'\le\Ker(\widehat{\chi})$ and $\chi=1$. 

Therefore, we have shown that $\Gamma$ is surjective and $F_\gamma L$ is isomorphic to a direct summand $B_1$ of $F\widehat{L}$ (as an ideal). Since, $\Z(B)\cong\Z(F_\gamma L)$ is a local algebra, $B_1$ is in fact a block of $F\widehat{L}$. Let $K:=\Ker(B_1)\le\pcore_{p'}(\widehat{L})=W$. Then $B_1$ is isomorphic to a faithful block $A$ of $G:=\widehat{L}/K$ (see \cite[Theorem~1.24]{habil}). For $\chi\in\Irr(A)$ we have $\pcore_{p'}(\Ker(\chi))=1$ (see \cite[Theoren~6.10]{Navarro}). In particular, the restriction $\chi_{W/K}$ is faithful. This implies that $Z:=\pcore_{p'}(G)\cong W/K$ is cyclic. The remaining properties follow easily.
\end{proof}

Let $B$ be a block of a group $G$ as in \autoref{main}. 
Let $P:=\pcore_p(G)$. Then $P\cong PZ/Z=\pcore_p(G/Z)$ and the Hall-Higman Lemma implies that $\C_G(P)=\Z(P)\times Z$. 
It follows that $G/PZ\le\Out(P)$ and $|G/Z|$ is bounded in terms of $|P|$ and therefore in terms of the defect of $B$. 
As a $p$-solvable group, $G$ has a $p$-complement $K\le G$. 
Then 
\[Z\le\cohom^2(G/Z,F^\times)\cong\cohom^2(K/Z,F^\times)\cong\pcore_{p'}(\cohom^2(K/Z,\CC^\times))\]
by \cite[Corollary~2.1.11]{Karpilovsky}. Hence, $|G|$ is bounded in terms of the defect of $B$. 

There are two important special cases of \autoref{main} which arise frequently:
\begin{itemize}
\item $B$ has a normal defect group, i.\,e. $P:=\pcore_p(G)\in\Syl_p(G)$. This happens for example whenever $B$ is a controlled block (for instance if the defect groups are abelian). We will apply \autoref{normal} in this situation.

\item $B$ is the principal block, i.\,e. $Z=1$ and $B=FG$. 
We will show in \autoref{dims} that the Morita equivalence class of $B$ determines $FG$ up to isomorphism. 
Here we are in a position to use \autoref{lem}.
\end{itemize}

Külshammer's paper~\cite{Kpsolv} provides not just any Morita equivalence, but an algebra isomorphism from $B$ to a matrix algebra over a twisted group algebra. This indicates our next result. 

\begin{Proposition}\label{dims}
Let $B$ be a block of $FG$ where $G$ is $p$-solvable. Then the Morita equivalence class of $B$ determines the dimensions of the simple modules of $B$ up to a common scalar. In particular, Morita equivalent principal blocks of $p$-solvable groups are isomorphic.
\end{Proposition}
\begin{proof}
We first determine the heights of the irreducible Brauer characters of $B$.
By the Fong-Swan Theorem, every irreducible Brauer character $\phi\in\IBr(B)$ lifts to $\widehat{\phi}\in\Irr(B)$ (see \cite[Theorem~10.1]{Navarro}). Hence, we may assume that the first $l(B)$ rows of the decomposition matrix of $B$ form an identity matrix. Let $C=(c_{\phi\mu})_{\phi,\mu\in\IBr(B)}$ be the Cartan matrix of $B$, and let $C^{-1}=(c'_{\phi\mu})$. Then Brauer's contribution numbers are given by $m_{\widehat{\phi}\widehat{\mu}}^1=p^dc'_{\phi\mu}\in\ZZ$ where $d$ is the defect of $B$ (see \cite[p. 14]{habil}). By \cite[Proposition~1.36(i)]{habil}, the Brauer characters $\phi$ of height $0$ are characterized by $p^dc'_{\phi\phi}\not\equiv 0\pmod{p}$. There is always at least one such character, say $\lambda\in\IBr(B)$. It follows from \cite[Proposition~1.36(ii)]{habil} that the height of any $\phi\in\IBr(B)$ equals the $p$-adic valuation of $p^dc'_{\phi\lambda}$. 

Now we consider the matrix $A:=(c_{\phi\mu}(p^dc'_{\mu\lambda})_p)_{\phi,\mu\in\IBr(B)}$. Let $v:=(\phi(1)_{p'}:\phi\in\IBr(B))$. Then \cite[Corollary~10.14]{Navarro} shows that $Av=p^dv$. 
Since $A$ is obtained from $C$ by scalar multiplication of columns, $A$ is a non-negative, indecomposable integer matrix. The Perron-Frobenius theory therefore implies that $A$ has only one positive eigenvector up to scalar multiplication (see \cite[Theorem~1.4.4]{Minc}). Consequently, the Morita equivalence class of $B$ determines $v$ up to a scalar. Since we also know the $p$-parts $\phi(1)_p$ from the heights up to a scalar, the first claim follows. 

For the second claim suppose that $A$ and $B$ are Morita equivalent principal blocks of $p$-solvable groups. Since $A$ and $B$ contain the $1$-dimensional trivial module, the dimensions of all simple $A$-modules and $B$-modules coincide by the first part of the proof. We denote these dimensions by $d=(d_1,\ldots,d_n)$. Let $P_1,\ldots,P_n$ be a set of representatives for the projective indecomposable $B$-modules up to isomorphism. Then $P_1^{d_1}\oplus\ldots\oplus P_n^{d_n}$ is isomorphic to the regular $B$-module.
The Morita equivalence between $A$ and $B$ induces an isomorphism between $A$ and the endomorphism algebra $\End_B(P_1^{s_1}\oplus\ldots\oplus P_n^{s_n})$ for some $s_1,\ldots,s_n\ge 1$. 
Let $C$ be the Cartan matrix of $A$ and $B$. Then the dimensions of the projective indecomposable modules of $A$ are given by $Cd$, but also by $Cs$ with $s=(s_1,\ldots,s_n)$. Since $C$ is invertible, it follows that $d=s$ and \[A\cong\End_B(P_1^{d_1}\oplus\ldots\oplus P_n^{d_n})\cong\End_B(B)\cong B.\qedhere\]
\end{proof}

In general, the dimensions of the simple modules of Morita equivalent blocks are not proportional. 
For example, by Scopes reduction, the principal $5$-blocks of $S_5$ and $S_6$ are Morita equivalent, but the dimensions are $1$, $1$, $3$, $3$ and $1$, $1$, $8$, $8$.

\begin{Lemma}\label{fixed}
Let $G=P\rtimes Q$ where $P$ is a $p$-group and $Q$ is a $p'$-group acting faithfully on $P$. Then the isomorphism type of $FG$ determines the fixed point algebras $FP^Q$ and $F\Z(P)^Q$ up to isomorphism. 
\end{Lemma}
\begin{proof}
If $FG$ is isomorphic to another group algebra $FH$, then we have seen in the proof of \autoref{lem} that there is an isomorphism preserving the augmentation. Let $e:=\frac{1}{|Q|}\sum_{a\in Q}a\in FG$. Then $e$ is an idempotent and $\dim FGe=FPe=|P|$. In particular, $e$ is primitive. Moreover, $e$ is the only primitive idempotent with augmentation $1$ up to conjugation. Consequently, $FG$ determines the algebra $eFGe=eFPe$ (this is the endomorphism ring of the trivial projective indecomposable module). 
For $x\in P$ we have
\[exe=\frac{1}{|Q|^2}\sum_{a,b\in Q}ab(b^{-1}xb)=\frac{\lvert\C_Q(x)\rvert}{|Q|}e\widetilde{x}\] 
where $\widetilde{x}:=\frac{1}{\lvert\C_Q(x)\rvert}\sum_{a\in Q}{axa^{-1}}\in FP^Q$. 
Let $x_1,\ldots,x_n\in P$ be a set of representatives for the $Q$-orbits on $P$. Then the elements $e\widetilde{x_1},\ldots,e\widetilde{x_n}$ form a basis of $eFPe$. 
In particular, $\dim eFPe=n=\dim FP^Q$. Since $ey=ye$ for $y\in FP^Q$, it follows that the map $FP^Q\to eFPe$, $y\mapsto ey$ is an isomorphism of $F$-algebras.

Now let $\K(FG)$ be the $F$-subspace of $FG$ spanned by the commutators $xy-yx$ for $x,y\in FG$. 
Obviously, $FG$ determines $\K(FG)$ and $\K(FG)$ is already spanned by the elements $gh-hg=gh-h(gh)h^{-1}$ for $g,h\in G$. This shows that
\[\K(FG)=\biggl\{\sum_{g\in G}\alpha_gg\in FG:\sum_{g\in C}\alpha_g=0\text{ for every conjugacy class $C$ of $G$}\biggr\}\]
(see \cite[(2)]{KProgress}).
Recall that the center $\Z(FG)$ is generated by the class sums of $G$ (as an $F$-space). Hence, $\Z(FG)\cap\K(FG)$ is generated by the class sums whose size is divisible by $p$. Since $Q$ acts faithfully, these are precisely the class sums not lying in $F\Z(P)$. Hence, the Brauer homomorphism with respect to $P$ yields an isomorphism 
\[\Z(FG)/\Z(FG)\cap\K(FG)\to F\Z(P)^Q.\qedhere\]
\end{proof}

In the situation of \autoref{fixed} the elementary divisors of the Cartan matrix of $FG$ are given by \[\lvert\C_P(x_1)\rvert,\ldots,\lvert\C_P(x_n)\rvert\] where $x_1,\ldots,x_n\in Q$ represent the conjugacy classes of $Q$. These numbers contain further information on the action of $Q$ on $P$. However, the action of $Q$ on $P$ is not uniquely determined by $FG$ as can be seen by Dade's example mentioned in the introduction.

According to \cite[p. 14]{Bessenrodtnewinv} it is an open question whether the group algebra $FG$ determines the commutativity of the Sylow $p$-subgroups of $G$. We give an affirmative answer for $p$-solvable groups.

\begin{Theorem}\label{abeldef}
Let $G$ be a $p$-solvable group with Sylow $p$-subgroup $P$. Then the isomorphism type of $FG$ determines if $P$ is abelian. If so, $FG$ also determines the isomorphism type of $P$.
\end{Theorem}
\begin{proof}
Using an augmentation preserving isomorphism as in \autoref{fixed}, we see that $FG$ determines the principal block $B$ of $G$ up to isomorphism. It is well-known that $B$ is isomorphic to the principal block of $G/\pcore_{p'}(G)$. Since $G$ is $p$-solvable, $G/\pcore_{p'}(G)$ has only one block and we may assume that $\pcore_{p'}(G)=1$ (see \cite[Theorem~10.20]{Navarro}). Then $N:=\pcore_p(G)$ is self-centralizing in $G$ by Hall-Higman. 
By \autoref{lem}, $FG$ determines if $N=P$. 
If $N<P$, then $P$ is non-abelian. Thus, we may assume that $P=N\unlhd G$. By Schur-Zassenhaus, $G=P\rtimes Q$ where $Q$ is a $p'$-group acting faithfully on $P$. Now from \autoref{fixed} we obtain $FP^Q$ and $F\Z(P)^Q$. Clearly, $P$ is abelian if and only if $\dim FP^Q=\dim F\Z(P)^Q$.
The last assertion follows from Bessenrodt~\cite[Theorem~2.1]{Bessenrodtabel} as mentioned earlier. 
\end{proof}

In general $FG$ does \emph{not} determine the commutativity of Sylow $q$-subgroups for $q\ne p$. For example the solvable groups $G=\texttt{SmallGroup}(1152,154124)$ and $H=\texttt{SmallGroup}(1152,154154)$ have the same multiset of irreducible character degrees, but $G$ has a non-abelian Sylow $2$-subgroup while $H$ has an abelian Sylow $2$-subgroup. Hence, for $p\ge 5$ (or $F=\CC$) the group algebras $FG$ and $FH$ are isomorphic. This answers a question of the first author raised in \cite{NavarroIndia}.

\begin{Corollary}
Let $B$ be the principal block of $FG$ where $G$ is $p$-solvable. Then the Morita equivalence class of $B$ determines if $B$ has abelian defect groups. If so, also the isomorphism type of the defect groups is determined. 
\end{Corollary}
\begin{proof}
By \autoref{main} we may assume that $B=FG$. By \autoref{dims}, the Morita equivalence class of $B$ determines the isomorphism type of $FG$. Now the claim follows from \autoref{abeldef}.
\end{proof}

We use the opportunity to propose a blockwise question for $p$-solvable groups.

\begin{Question}\label{Q}
Let $B$ be a $p$-block of a $p$-solvable group with fusion system $\mathcal{F}$. Does the Morita equivalence class of $B$ (or the isomorphism type) determine $\lvert\pcore_p(\mathcal{F})\rvert$? Here, $\pcore_p(\mathcal{F})$ is just the group $\pcore_p(G)$ in the situation of \autoref{main} (see \cite[Theorem~7.18]{habil}).
\end{Question}

Note that a block neither “knows” if its defect group is normal nor if it is the principal block. For example the principal $2$-block of $FS_3$ is isomorphic to a non-principal $2$-block of $FC_6$.  

The following lemma is already known for $\FF_2$. However, we are not aware of a proof over an algebraically closed field.

\begin{Lemma}\label{mip}
Let $G$ be a group of order $32$. Then the isomorphism type of $FG$ determines $G$ up to isomorphism.
\end{Lemma}
\begin{proof}
In addition to the invariants listed in Propositions \ref{lem} and \ref{normal}, also the isomorphism types of $\Z(G)$ and $G/G'$ are determined by $FG$ (see \cite[Lemma~14.2.7]{PassmanASGR}).
After comparing these invariants we are left with three pairs of groups: $(10,14)$, $(30,31)$ and $(32,35)$ where the numbers represent the indices in the small group library. To distinguish these groups we can consider the cohomology rings $\cohom^*(G,F)\cong F\otimes_{\FF_p}\cohom^*(G,\FF_p)$ which are given in \cite[Appendix]{CTVZ}. It can be seen that the minimal number of generators of these rings are different for every pair of groups above.
\end{proof}

\begin{Proposition}\label{2def5}
Let $B$ be a $2$-block of defect $5$ of $FG$ where $G$ is ($2$-)solvable. Then the Morita equivalence class of $B$ determines the defect groups of $B$ up to isomorphism.
\end{Proposition}
\begin{proof}
We assume that $G$ is given as in \autoref{main}. Let $D$ be a defect group of $B$. By \autoref{coexp1}, we may assume that $D$ has exponent $4$ or $8$. Let $P:=\pcore_2(G)$, and let $K\le G$ be a $2$-complement.
Then $K/Z$ acts faithfully on $P/\Phi(P)$. Since $P/\Phi(P)$ is elementary abelian of rank at most $4$, we deduce that $K/Z\in\{1,C_3,C_5,C_7,C_3^2,C_{15},C_7\rtimes C_3\}$. 

Assume first that $|K/Z|\ne 9$. Then $K/Z$ has trivial Schur multiplier and $Z=1$. This gives $B=FG$ and $l(B)=l(G/P)$. In particular, $l(B)=1$ if and only if $B=FD$. 
If, on the other hand, $|K/Z|=9$, then the Schur multiplier is $C_3$. If $Z\cong C_3$, then $K$ is non-abelian and $B$ is one of the two non-principal blocks of $G$.
Now we run through all possible groups $G$ of order $32a$ where $a\in\{1,3,5,7,9,15,21,27\}$ with GAP~\cite{GAP48} and compute the invariants exponent, rank, $k(B)$ and Cartan matrix. It turns out that $Z\ne 1$ only if $l(B)=1$. These cases can be distinguished from the nilpotent blocks (and among each other) by comparing $k(B)$. It remains to handle the case $Z=1$. 
By \autoref{mip}, we may assume that $K\ne 1$. It turns out that the only difficult groups all have $|K|=3$. 
Here the defect groups can be distinguished by using Propositions \ref{lem} and \ref{normal}.
\end{proof}

\begin{Proposition}
Let $B$ be a $3$-block of defect $4$ of $\mathcal{O}G$ where $G$ is $3$-solvable. Then the Morita equivalence class of $B$ determines the defect groups of $B$ up to isomorphism.
\end{Proposition}
\begin{proof}
We assume that $G$ is given as in \autoref{main} (the result holds over $\mathcal{O}$ too). Since $G$ is not necessarily in the small group library, we need a more careful analysis than in \autoref{2def5}. 
Let $D$ be a defect group of $B$. As usual, we may assume that $\exp(D)\le 9$ by \autoref{coexp1}. 
There are only two groups of order $81$ and exponent $3$, namely $C_3^4$ and $3^{1+2}_+\times C_3$. Since they differ by their rank, we may assume that $\exp(D)=9$. 
Since the height zero conjecture holds for $p$-solvable groups, we may assume further that $D$ is non-abelian. There are eight such groups, five of them have rank $2$ and the remaining three have rank $3$.
In \autoref{tab} we refer to the id in the small groups library. 
As mentioned in the introduction, we may assume that $B$ is non-nilpotent.
Since $\pcore_3(G)$ is radical, it is easy to see that $\pcore_3(G)\in\{C_3^3, 3^{1+2}_+,D\}$. 
By comparing the subgroups of $\GL(3,3)$ and $\Aut(D)$ we obtain the possibilities for $G/\F(G)$ where $\F(G)$ denotes the Fitting subgroup of $G$. In all cases $G/\F(G)$ has elementary abelian Schur multiplier. It follows that $|Z|\le 2$, since $Z$ is cyclic. If $|Z|=2$, then $B$ is the unique non-principal block of $G$. Hence, $l(B)=l(G/\F(G))$ if $Z=1$ and $l(B)=l(G/\pcore_3(G))-l(G/\F(G))$ if $|Z|=2$ (here $G/\pcore_3(G)$ is a double cover of $G/\F(G)$). These results are summarized in \autoref{tab}.

\begin{table}[ht]
\begin{center}
\caption{Some $3$-blocks of defect $4$ of $3$-solvable groups}\label{tab}
\begin{tabular}{ccccc}
id of $D$&rank of $D$&$\pcore_3(G)$&$G/\F(G)$&$l(B)$\\\hline\hline
3&3&$D$&$C_2$&2\\
&&$D$&$C_2^2$&4, 1\\
&&$C_3^3$&$\SL(2,3)$&3\\
&&$C_3^3$&$\GL(2,3)$&6\\\hline
4&2&$D$&$C_2$&2\\\hline
7&3&$D$&$C_2$&2\\
&&$D$&$C_2^2$&4, 1\\
&&$C_3^3$&$A_4$&2, 1\\
&&$C_3^3$&$S_4$&4, 2\\
&&$C_3^3$&$A_4\times C_2$&4, 2\\
&&$C_3^3$&$S_4\times C_2$&8, 4, 2, 1\\ 
&&$C_3^3$&$C_{13}\rtimes C_3$&5\\
&&$C_3^3$&$(C_{13}\rtimes C_3)\times C_2$&10\\
&&$3^{1+2}_+$&$\SL(2,3)$&3\\
&&$3^{1+2}_+$&$\GL(2,3)$&6\\\hline
8&2&$D$&$C_2$&2\\
&&$D$&$C_2^2$&4, 1\\
&&$3^{1+2}_+$&$\SL(2,3)$&3\\
&&$3^{1+2}_+$&$\GL(2,3)$&6\\\hline
9&2&$D$&$C_2$&2\\
&&$D$&$C_2^2$&4, 1\\
&&$3^{1+2}_+$&$\SL(2,3)$&3\\
&&$3^{1+2}_+$&$\GL(2,3)$&6\\\hline
10&2&$D$&$C_2$&2\\\hline
13&3&$D$&$C_2$&2\\
&&$D$&$C_2^2$&4, 1\\\hline
14&2&$D$&$C_2$&2\\
&&$D$&$C_4$&4\\
&&$D$&$C_2^2$&4, 1\\
&&$D$&$C_8$&8\\
&&$D$&$D_8$&5, 2\\
&&$D$&$Q_8$&5\\
&&$D$&$SD_{16}$&7
\end{tabular}
\end{center}
\end{table}

In the last column there are sometimes several possibilities according to the chosen double cover of $G/\F(G)$ (all listed possibilities actually occur). If $l(B)=10$, then $D$ is uniquely determined. 
Putting this case aside, the only large groups not contained in the small group library correspond to line 11 in \autoref{tab}. These groups are double covers of $\texttt{SmallGroup}(1296,3490)$ and can be constructed in GAP.
For any finite group $H$ and $i\ge 0$ let $K_i(H)$ be the number of $H$-conjugacy classes in $\{h^{3^i}:h\in H\}$. Then the dimension of the $i$-th Külshammer ideal of $B$ is given by $K_i(G)$ if $Z=1$ and $K_i(G)-K_i(G/Z)$ if $|Z|=2$ (these numbers are invariant under Morita equivalence by \cite[Corollary~5.3]{HHKM}).
Now using \autoref{normal} and \autoref{lem} we can distinguish the defect groups up to three remaining pairs: $(648,75)\leftrightarrow(648, 82)$, $(648,531)\leftrightarrow(648,532)$, $(1296,2889)\leftrightarrow(1296, 2890)$ (small group ids). For all pairs the Sylow ids are $8$ and $9$. For the first pair we have non-principal blocks with normal defect group and only one simple module. For the last two pairs we have principal blocks with non-normal defect groups. To handle these cases we construct the basic algebra of $B$ over a finite splitting field in MAGMA~\cite{Magma}. Then we can compare minimal projective resolutions for instance. The decomposition of the projective modules in such resolutions do not depend on the size of the field. In this way we complete the proof.
\end{proof}

\section*{Acknowledgment}
The authors like to thank Bettina Eick, Karin Erdmann, David Green, Burkhard Külshammer, Pierre Landrock and Leo Margolis for answering some questions. The first author is partially supported by the Spanish Ministerio de Educaci\'on y Ciencia Proyectos  MTM2016-76196-P and Prometeo II/Generalitat Valenciana. The second author is supported by the German Research Foundation (project SA \mbox{2864/1-1}).

\end{document}